\documentclass[12pt]{amsart}
\usepackage{graphicx}
\usepackage{amsmath,amsthm, amsfonts,amssymb, stmaryrd,yfonts,pxfonts,pifont,eufrak, bbm}
\newtheorem{Cor}{Corollary}
 \newtheorem{Lemma}{Lemma}
 
 \newtheorem{ex}{Example}
 \newtheorem{p}{Proposition}
 \theoremstyle{definition}
 
 \theoremstyle{remark}
 \newtheorem{Remark}[Lemma]{Remark}
 \numberwithin{equation}{subsection}

\begin{document}
\title[INDUCED DYNAMICS OF NON-AUTONOMOUS DISCRETE DYNAMICAL SYSTEMS]{INDUCED DYNAMICS OF NON-AUTONOMOUS DISCRETE DYNAMICAL SYSTEMS}%
\author{PUNEET SHARMA}
\address{Department of Mathematics, I.I.T. Jodhpur, Old Residency Road, Ratanada, Jodhpur-342011, INDIA}%
\email{puneet.iitd@yahoo.com}%



\keywords{hyperspaces, non-autonomous dynamical systems, transitivity, weakly
mixing, topological mixing, topological entropy, Li-Yorke chaos}

\begin{abstract}
In this paper, we investigate the dynamics on the hyperspace induced by a non-autonomous dynamical system $(X,\mathbb{F})$, where the non-autonomous system is  generated by a sequence $(f_n)$ of continuous self maps on $X$. We relate the dynamical behavior of the induced system on the hyperspace with the dynamical behavior of the original system $(X,\mathbb{F})$. We derive conditions under which the dynamical behavior of the non-autonomous system extends to its induced counterpart(and vice-versa). In the process, we discuss properties like transitivity, weak mixing, topological mixing, topological entropy and various forms of sensitivities. We also discuss properties like equicontinuity, dense periodicity and Li-Yorke chaoticity for the two systems. We also give examples when a dynamical notion of a system cannot be extended to its induced counterpart (and vice-versa).
\end{abstract}
\maketitle

\section{INTRODUCTION}

For many years, dynamical systems have been studied to investigate many of the physical or natural phenomenon occurring in nature. Using dynamical systems, long term behavior of various natural phenomenon have been predicted to sufficient accuracy and many of the underlying processes have been investigated using the theory of dynamical systems\cite{de,str}. In many cases, the structure of underlying space is a pivotal factor in determining the dynamical behavior of the system. This has resulted in further investigations in the field of topological dynamics and a lot of work in this area has already been done\cite{bc,bs,de}. However, most of the phenomenon occurring in nature arise collectively as union of several individual components and hence set valued dynamics plays an important role in understanding any of these phenomenon. Such an approach has found applications in various branches of sciences and engineering\cite{klaus,seb,yang}. Thus there was a strong need to develop and understand the dynamical behavior of the induced set valued systems. As a result, many of the natural questions relating the dynamical behavior of a system and its set valued counterpart have been raised and answered\cite{ba,pa1,pa2}. However, most of the investigations have been made when the rule determining the underlying system is time-invariant. Such an approach fails to investigate the dynamics of a general system as the governing rule of a natural phenomenon may change with time. Thus there is a need to understand the dynamics of the induced system when the underlying system is non-autonomous in nature. In this paper, we investigate the dynamical behavior of the induced system, when induced by a non-autonomous system. We prove that many of the results for the non-autonomous case are analogous extensions of the autonomous case. We derive conditions under which the dynamical behavior of a system is extended to the hyperspace(and vice-versa). In the process, we discuss properties like dense periodicity, various forms of mixing and sensitivity, equicontinuity and Li-Yorke chaoticity for the two systems. We establish our results when the hyperspace is endowed with a general admissible hyperspace topology. Before we move further, we provide some of the basic concepts required.

\subsection{Dynamical Systems}

Let $(X,d)$ be a compact metric space and let $\mathbb{F}= \{f_n: n
\in \mathbb{N}\}$ be a family of continuous self maps on $X$. Let $(X,\mathbb{F})$ denote the non-autonomous system generated by the family $\mathbb{F}$ via the rule $x_{n}= f_n(x_{n-1})$. For any point
$x\in X$, $\{ f_n \circ f_{n-1} \circ \ldots \circ f_1(x) :
n\in\mathbb{N}\}$ defines the orbit of $x$. The objective of study
of any non-autonomous dynamical system is to investigate the orbit of
an arbitrary point $x$ in $X$. For notational convenience, let
$\omega_n(x) = f_n\circ f_{n-1}\circ \ldots \circ f_1(x)$ denote the
state of the system after $n$ iterations. \\

A point $x$ is called \textit{periodic} for $\mathbb{F}$ if there exists $n\in\mathbb{N}$ such that $\omega_{nk}(x)=x$ for all $k\in
\mathbb{N}$. The least such $n$ is known as the period of the point $x$. The system $(X,\mathbb{F})$ is equicontinuous at a point $x\in X$ if for each $\epsilon>0$, there exists $\delta>0$ such that $d(x,y)<\delta$ implies $d(\omega_n(x),\omega_n(y))<\epsilon$ for all $n\in\mathbb{N},~~y\in X$. The system is called equicontinuous if it is equicontinuous at each point of $X$. The system $(X,\mathbb{F})$ is uniformly equicontinuous if for each $\epsilon>0$, there exists $\delta>0$ such that $d(x,y)<\delta$ implies $d(\omega_n(x),\omega_n(y))<\epsilon$ for all $n\in\mathbb{N},~~x,y\in X$. The system $(X,\mathbb{F})$ is \textit{transitive} (or $\mathbb{F}$ is transitive) if for each pair of open sets $U,V$ in $X$, there exists $n \in \mathbb{N}$ such that $\omega_n(U)\bigcap V\neq \phi$. Let $\mathbb{F}_n=\{f_{kn+1}\circ f_{kn+2}\circ\ldots\circ f_{(k+1)n}: k\in\mathbb{Z}^+\}$. The system $(X,\mathbb{F})$ is called $n$-transitive if the system $(X,\mathbb{F}_n)$ is transitive. If $(X,\mathbb{F})$ is $n$-transitive for each $n\in\mathbb{N}$, then the system is called totally transitive. The system $(X,\mathbb{F})$ is said to be \textit{weakly mixing} if for any two pairs $U_1, U_2$ and $V_1, V_2$ of non-empty open subsets of $X$, there exists a natural number $n$ such that $\omega_n(U_i) \bigcap V_i \neq \phi$ for $i=1,2$. Equivalently, we say that the system is weakly mixing if $\mathbb{F}\times\mathbb{F}$ is transitive. The system $(X,\mathbb{F})$ is said to be \textit{weakly mixing of order $k$} if for any collection $U_1, U_2,\ldots,U_k$ and $V_1, V_2,\ldots,V_k$ of non-empty open subsets of $X$, there exists a natural number $n$ such that $\omega_n(U_i) \bigcap V_i \neq \phi$ for $i=1,2,\ldots,k$. The system is said to be \textit{topologically mixing} if for every pair of non-empty open sets $U, V$ in $X$, there exists a natural number $K$ such that $\omega_n(U) \bigcap V \neq \phi$ for all $n \geq K$. The system is said to be \textit{sensitive} if there exists a $\delta>0$ such that for each $x\in X$ and each neighborhood $U$ of $x$, there exists $n\in \mathbb{N}$ such that $diam(\omega_n(U))>\delta$. If there exists $K>0$ such that $diam(\omega_n(U))>\delta$, $~~\forall n\geq K$, then the system is \textit{cofinitely sensitive}. A system $(X,\mathbb{F})$ is \textit{expansive}($\delta$-expansive) if for any pair of distinct elements $x,y \in X$, there exists $k \in \mathbb{Z}^+$ such that $d(\omega_k(x), \omega_k(y))>\delta$. A set $S$ is said to be $\delta$-\textit{scrambled} if for any distict pair of points $x,y\in S$, $\limsup\limits_{n\rightarrow \infty} d(\omega_n(x),\omega_n(y))>\delta$ but $\liminf\limits_{n\rightarrow \infty} d(\omega_n(x),\omega_n(y))=0$. A system is called Li-Yorke sensitive if there exists $\delta>0$ such that for each $x\in X$ and each neighborhood $U$ of $x$, there exists $y\in U$ such that $\{x,y\}$ is a $\delta$-scrambled set. A system $(X,\mathbb{F})$ is said to be \textit{Li-Yorke chaotic} if it contains an uncountable scrambled set.  A dynamical system $(X,f)$ has \textit{chaotic dependence on initial conditions} if for any $x \in
X$ and any neighborhood $U$ of $x$ there exists $y\in U$ such that $\limsup \limits_{n \rightarrow \infty} d(f^n(x),
f^n(y))> 0$ but $\liminf \limits_{n \rightarrow \infty} d(f^n(x),f^n(y)) =0$. In case the $f_n$'s coincide, the above definitions coincide with the known notions of an autonomous dynamical system. See \cite{bc,bs,de} for details.\\

We now define the notion of \textit{topological entropy} for a
non-autonomous system $(X,\mathbb{F})$.\\

Let $X$ be a compact space and let $\mathcal{U}$ be an open cover of
$X$. Then $\mathcal{U}$ has a finite subcover. Let $\mathcal{L}$ be
the collection of all finite subcovers and let $\mathcal{U}^*$ be
the subcover with minimum cardinality, say $N_{\mathcal{U}}$. Define
$H(\mathcal{U}) = log N_{\mathcal{U}} $. Then $H(\mathcal{U})$ is
defined as the \textit{entropy} associated with the open cover
$\mathcal{U}$. If $\mathcal{U}$ and $\mathcal{V}$ are two open
covers of $X$, define, $\mathcal{U} \vee \mathcal{V} = \{ U \bigcap
V : U \in \mathcal{U}, V \in \mathcal{V} \}$. An open cover $\beta$
is said to be refinement of open cover $\alpha$ i.e. $\alpha \prec
\beta$, if every open set in $\beta$ is contained in some open set
in $\alpha$. It can be seen that if $\alpha \prec \beta$ then
$H(\alpha) \leq H(\beta)$. For a self map $f$ on $X$, $f^{-1}
(\mathcal{U}) = \{ f^{-1} (U) : U \in \mathcal{U} \}$
is also an open cover of $X$. Define,\\

\centerline{$h _{\mathbb{F}, \mathcal{U}} = \limsup \limits_{k
\rightarrow \infty} \frac{H( \mathcal{U} \vee
\omega_1^{-1}(\mathcal{U}) \vee \omega_2^{-1}(\mathcal{U}) \vee
\ldots \vee \omega_{k-1}^{-1}(\mathcal{U}))}{k}$} \vskip .25cm

Then $\sup h _{\mathbb{F}, \mathcal{U}}$, where $\mathcal{U}$ runs
over all possible open covers of $X$ is known as the
\textit{topological entropy of the system $(X,\mathbb{F})$} and is
denoted by $h(\mathbb{F})$. In case the maps $f_n$ coincide, the
above definition coincides with the known notion of topological
entropy. See \cite{bc,bs} for details.

%
%

\subsection{Hyperspaces} Let $(X,\tau)$ be a Haudorff topological space. A hyperspace associated with $(X,\tau)$ is a pair $(\Psi,\Delta)$  where $\Psi$ comprises of a subfamily of all non-empty closed subsets of $X$ and $\Delta$ is a topology on $\Psi$ generated using topology on $X$. The set $\Psi$ may comprise of all compact subsets of $X$ or all compact-connected subsets of $X$ or all closed subsets of $X$. A hyperspace topology is called admissible if the map $x \rightarrow \{ x \}$ is continuous. More generally, if $\Psi$ and $\Delta$ are fixed, the induced space $(\Psi, \Delta)$ is called the hyperspace generated from the space $(X, \tau)$. Let $CL(X)$ and $\mathcal{K}(X)$ denote set of all non-empty closed and non-empty compact subsets of $X$ respectively. We now give some of the standard hyperspace topologies.

Let $I$ be a finite index set and for all such $I$, let $\{ U_i : i
\in I \}$ be a collection of open subsets of $X$. Define for each
such collection of open sets,

$<U_{i}>_{i \in I}$ = $ \{E \in CL(X) :E \subseteq \bigcup_{i \in I}
U_{i}$ and E $\bigcap U_{i} \neq \phi \textrm{ } \forall i \}$

The topology generated by such collections is known as the
\textit{Vietoris topology}.

Let $(X,d)$ be a metric space. For any two closed subsets $A_1, A_2$
of $X$, define,

\centerline{$ d_H (A_1, A_2) = \inf \{ \epsilon >0 : A \subseteq
S_{\epsilon} (B) \text{ and } B \subseteq S_{\epsilon} (A) \} $}

It is easily seen that $d_H$ defined above is a metric on $CL(X)$
and is called \textit{Hausdorff metric} on $CL(X)$. This metric
preserves the metric on $X$, i.e. $d_H(\{x\}, \{y\}) = d(x,y)$ for
all $x,y \in X$. The topology generated by this metric is known as
the \textit{Hausdorff metric topology} on $CL(X)$ with respect to
the metric $d$ on $X$.

It is known that the Hausdorff metric topology equals the Vietoris
topology if and only if the space $X$ is compact.

Let $\Phi$ be a subfamily of the collection of all non-empty closed
subsets of $X$. The \textit{Hit and Miss topology} determined by the
collection $\Phi$ is the topology having subbasic open sets of the
form $U^-$ where $U$ is open in $X$ and $(E^c)^{+}$ with $E \in
\Phi$. As a terminology, $U$ is called the hit set and any member
$E$ of $\Phi$ is referred as the miss set.

For a metric space $(X,d)$ and a given collection $\Phi$ of closed
subsets of $X$, the \textit{Hit and Far Miss topology} or
\textit{Proximal Hit and Miss Topology} determined by the collection
$\Phi$ is the topology having subbasic open sets of the form $U^-$
where $U$ is open in $X$ and $(E^c)^{++}$ with $E \in \Phi$.

Here the collection hits each open set $U$ and far misses the
complement of each member of $\Phi$ and hence forms a hit and far
miss topology. 

A typical member of the base for the \textit{Lower Vietoris
topology} on the hyperspace $CL(X)$ consists of the set, each of
whose elements intersect or $hit$ finitely many open sets $U$, i.e.
a typical basic open set is the intersection of finitely many $U^-$.
The Lower Vietoris topology is the smallest topology on the
hyperspace containing all the sets $U^-$ where $U$ is open in $X$.

A typical basic open set for the \textit{Upper Vietoris topology} on
the hyperspace $CL(X)$ is of the form $U^+$ where $U$ is open in
$X$. Thus, given a closed set $C$, a typical member of the base in
the Upper Vietoris topology is the set whose elements are the
elements of the hyperspace disjoint from the closed set $C$.

It is observed that the Vietoris topology equals the join of Upper Vietoris and Lower
Vietoris topology, and is infact an example of a hit and miss
topology. More generally, it is known that any admissible hyperspace topology is of hit-and-Miss or Hit and Far Miss type. See\cite{be,mi,nai} for details.

%

\section{Main Results}

Let $(X,\mathbb{F})$ be a non-autonomous topological dynamical system and let $\Psi\subset\mathcal{K}(X)$ be a hyperspace admissible with $\mathbb{F}$, i.e. for any $A\in\Psi$,$\omega_k(A)\in\Psi$ for all $k\in\mathbb{N}$. Then for any initial seed $A_0\in\Psi$, the non autonomous system $(X,\mathbb{F})$ induces a non-autonomous system $(X,\overline{\mathbb{F}})$ on the hyperspace via the relation $A_{n}=f_n(A_{n-1})=\omega_n(A_0)$. Let $\overline{\omega_k}(A)$ denote the state of the point $A$(in the hyperspace) after $k$ iterations. Let the hyperspace be endowed with a topology such that the each of the induced non-autonomous system $(X,\overline{\mathbb{F}})$ is continuous. It is intuitive to question the relation between the dynamical behavior of the non-autonomous system and its induced counterpart. For example, if a given non-autonomous system is transitive/weakly mixing/topologically mixing, what can be concluded about its induced counterpart (and vice-versa). If a given non-autonomous system exhibits sensitivity/strong sensitivity/Li-Yorke sensitivity, what can be concluded about its induced counterpart (and vice-versa). What dynamics does the induced system exhibit when the original system is equicontinuous? Such questions have been raised and answered in case of an autonomous system\cite{ba,pa1,pa2}. We prove that answers to many of the question remain the same in a non-autonomous setting and hence most of the results obtained for the autonomous case can be extended analogously to the non-autonomous case. We prove that the induced system is topologically mixing if and only if the original system is topologically mixing. We prove that strong sensitivity is also equivalent for the two systems under consideration.  We prove that if $\mathbb{F}$ is commutative then a system exhibits weak mixing of all orders if and only if the induced system is weak mixing. We extend our studies to properties like dense periodicity, transitivity, equicontinuity, uniform equicontinuity, various notions of sensitivities and Li-Yorke chaos. We now establish the stated results.

\begin{p}
Let $\mathcal{F} (X) \subseteq \Psi$ and $\Psi$ be
endowed with any admissible hyperspace topology. Then, $(X,\mathbb{F})$ has dense set of periodic points $\Rightarrow$ $(\Psi,\overline{\mathbb{F}})$ has dense set of periodic points.
\end{p}

\begin{proof}
Let the  hyperspace $\Psi$ be endowed with an admissible topology $\Delta$. As every admissible topology on the hyperspace is of hit and
miss or hit and far miss type, let the topology $\Delta$ be determined by the collection $\mathcal{C}$. Let $\mathcal{U}$ be a non empty basic open set in the hyperspace. Then $\mathcal{U}$ hits finitely many open sets, say $V_1, V_2, \ldots, V_{n}$ and misses(far misses)
finitely many elements of $\mathcal{C}$ say, $C_1, C_2, \ldots, C_{m}$. Let $ C = \bigcup C_j$. Thus each $ W_i = V_i \bigcap C^c$ is non-empty, open in $X$. As periodic points are dense for $(X,\mathbb{F})$, there exists $x_i\in W_i$  and an integer $r_i\in\mathbb{N}$ such that for any $i=1,2,\ldots,n$, $\omega_{r_i k}(x_i)=x_i~~\forall k\in\mathbb{N}$. Let $l= lcm\{r_1,r_2,\ldots,r_n\}$. Then the set $\{x_1, x_2, ... x_n\}$ is periodic with period $l$. As $\{x_1,x_2,\ldots,x_n\}\in \mathcal{U}$, the hyperspace has a dense set of periodic points.
\end{proof}

\begin{Remark}
The above proof establishes that if the hyperspace contains all finite sets then the denseness of periodic points is extended to the hyperspace when the hyperspace is endowed with any admissible hyperspace topology. The proof is analogous to the autonomous case and is a natural extension of the result established in \cite{pa1}. We now give an example to show that the converse of the above does not hold good.
\end{Remark}

\begin{ex}
Let $\Sigma$ be the space of all one-sided sequences of $0$ and $1$ and let $\phi:\Sigma\rightarrow\Sigma$ be defined as $f(x=(x_1x_2\ldots))=x+(100\ldots)$ where addition is performed with carry to the right. It is known that $(\Sigma,\phi)$ does not contain any periodic point but the cylinder sets $[x_1x_2\ldots x_k]$ are periodic and hence the hyperspace $(\mathcal{K}(\Sigma),\overline{\phi})$ has dense set of periodic points\cite{ba}. Let $I$ be the identity operator on $\Sigma$ and let the non-autonomous system $(X,\mathbb{F})$ defined by the family $\mathbb{F}=\{I,\phi,I,\phi,\ldots,I,\phi,\ldots\}$. Then, $(X,\mathbb{F})$ fails to contain any periodic point but $(\mathcal{K}(\Sigma),\overline{\mathbb{F}})$ exhibits dense set of periodic points. Thus the converse of the above result fails to hold for the non-autonomous system.
\end{ex}

\begin{p}
If there exists a base $\beta$ for the topology on
$X$ such that $U^+$ is non empty and $U^+ \in \Delta$ for every $U
\in \beta$, then $(\Psi,\overline{\mathbb{F}})$ is transitive $\Rightarrow$
$(X,\mathbb{F})$ is transitive.
\end{p}

\begin{proof}
Let $U$ and $V$ be any two non-empty open sets in $X$. As $\beta$ forms a base for topology on $X$, there exists $U_1, V_1 \in \beta $
such that $U_1 \subseteq U$ and $V_1 \subseteq V$. As $U_1 ^+$ and $V_1 ^+$ are non empty open sets in the hyperspace $\Psi$ and $(\Psi,\overline{\mathbb{F}})$ is transitive, there exists $A\in U_1^+$  and $k\in\mathbb{N}$ such that $\overline{\mathbb{\omega}}_k(A)\in V_1 ^+$ or $\omega_k(A)\in V_1^+$. Consequently for any $a\in A$, $a\in A\subset U_1\subset U$ and $\omega_k(a)\in V_1\subset V$ and hence $(X,\mathbb{F})$ is transitive.
\end{proof}

\begin{Remark}
The above proof establishes the transitivity of the original function from the transitivity of the induced function. The result is an analogous extension from the autonomous case and does not come up as a surprise in the non-autonomous setting. It may be noted that identical arguments of the proof establish the n-transitivity of $(X,\mathbb{F})$ from n-transitivity of the induced system and hence total transitivity on the hyperspace implies total transitivity of $(X,\mathbb{F})$. Thus we get the following corollary.
\end{Remark}

\begin{Cor} If there exists a base $\beta$ for the topology on $X$ such that $U^+$ is non empty and $U^+ \in \Delta$ for every $U
\in \beta$, then $(\Psi, \overline{\mathbb{F}})$ is totally transitive $\Rightarrow~~$ $(X,\mathbb{F})$ is totally transitive.
\end{Cor}


\begin{p} Let $\mathcal{F} (X) \subseteq \Psi.$ If $\mathbb{F}$ is weakly mixing of all orders then $\overline{\mathbb{F}}$ is weak mixing. Further, if $\mathbb{F}$ is commutative and there exists a base $\beta$ for topology on $X$ such that $U^+ \in \Delta$ for every $U \in \beta$ then $\overline{\mathbb{F}}$ is weak mixing $\Rightarrow\mathbb{F}$ is weakly mixing of all orders.
\end{p}

\begin{proof}
Let the topology $\Delta$ on the hyperspace be determined by the collection $\mathcal{C}$ and let $\mathcal{U}_1$, $\mathcal{U}_2$, $\mathcal{V}_1$,
$\mathcal{V}_2$ be non-empty open sets in the hyperspace. Let $W_{11}, W_{21}, \ldots W_{n_1 1}$; $W_{12},
W_{22}, \ldots W_{r_1 2}$; $R_{11}, R_{21}, \ldots R_{n_1 1}$; $R_{12}, R_{22}, \ldots R_{r_1 2}$ define the collection of hit sets and $T_{11}, T_{21}, \ldots T_{m_1 1}$; $T_{12}, T_{22}, \ldots T_{s_1 2}$; $S_{11}, S_{21}, \ldots S_{m_1 1}$; $S_{12}, S_{22}, \ldots S_{s_1 2}$ define the collection of miss sets for $\mathcal{U}_1$, $\mathcal{U}_2$, $\mathcal{V}_1$, $\mathcal{V}_2$ respectively. Let $ T_i = \bigcup \limits_{j} T_{j i}$, $ S_i = \bigcup \limits_{j} S_{j i}$, $P_{j}^{i} = W_{j i} \bigcap T_i ^c$ and $Q_{j}^{i} = R_{j
i} \bigcap S_i ^c$. As $\mathcal{U}_1$, $\mathcal{U}_2$, $\mathcal{V}_1$, $\mathcal{V}_2$ are non-empty, each of $P_{j}^{i}, Q_{j}^{i}$ are non-empty open sets. Further as $\mathbb{F}$ is weakly mixing of all orders, there exists $k \in \mathbb{N}$ such that $\omega_k
(P_{j}^{i}) \bigcap Q_{j}^{i} \neq \phi$,  $\forall ~ i, j$. Choose $x_{j}^{i} \in P_{j}^{i}$ such that $ \omega_k(x_{j}^{i}) \in Q_{j}^{i}$.
Then $ A_i = \{ x_{j}^{i}\}_ j \in \mathcal{U}_{i}$ such that $\overline{\omega_k}(A_i) \in \mathcal{V}_{i}$. Hence $\overline{\mathbb{F}}$ is
weakly mixing.

Conversely, let $U_1, U_2,\ldots,U_m$ and $V_1, V_2,\ldots, V_m$ be non-empty open sets in $X$.  As $\beta$ is
the base for the topology on $X$, $\exists$ $U_{11},U_{22},\ldots,U_m$ and  $V_{11},V_{22},\ldots,V_{mm} \in \beta $ such that $U_{ii} \subseteq U_i$ and $V_{ii}\subseteq V_i$ for $i=1,2,\ldots,m$. As $f_n$'s commute in the original system, $\overline{f_n}$ commutes on the hyperspace and hence by \cite{pm} the induced system is $n$-transitive for any $n\in\mathbb{N}$. Hence for open sets $\{U_{ii}^+:i=1,2,\ldots,m\}$ and $\{V_{ii}^+: i=1,2,\ldots,m\}$, there exists $k \in \mathbb{N}$ such that $\overline{\omega}_k (U_{ii} ^+) \bigcap V_{ii} ^+ \neq \phi$ for $i=1,2,\ldots,m$  which implies $\omega_k(U_{ii})\cap V_{ii}\neq\phi$ and the proof is complete.
\end{proof}

\begin{Remark}
The above proof establishes that for a commutative family $\mathbb{F}$, weak mixing on the hyperspace is equivalent to weak mixing of all orders in the original system. It may be noted that the forward part of the proof does not require commutativity of the family $\mathbb{F}$ and arguments similar to the converse establish weak mixing of order $m$ in the original system from weak mixing of order $m$ of the induced system without using the commutativity of $\mathbb{F}$. Thus the result more generally establishes that the original system is weak mixing of all orders if and only if the induced system is weak mixing of all orders. Further, as weak mixing of all orders is equivalent to weak mixing of second order for a commutative family $\mathbb{F}$, the result establishes equivalence of weak mixing for the two systems when the non-autonomous system is induced by a commutative family. Hence we get the following corollaries.
\end{Remark}

\begin{Cor} Let $\mathcal{F} (X) \subseteq \Psi.$ If there exists a base $\beta$ for topology on $X$ such that $U^+ \in \Delta$ for every $U \in \beta$, then $(X,\mathbb{F})$ is weakly mixing of all orders if and only if $(\Psi,\overline{\mathbb{F}})$ is weak mixing of all orders.
\end{Cor}

\begin{Cor} Let $\mathcal{F} (X) \subseteq \Psi.$ If $\mathbb{F}$ is commutative and there exists a base $\beta$ for topology on $X$ such that $U^+ \in \Delta$ for every $U \in \beta$, then $(X,\mathbb{F})$ is weakly mixing if and only if $(\Psi,\overline{\mathbb{F}})$ is weak mixing.
\end{Cor}

\begin{Remark}
For the results derived so far, it may be noted that the proofs establish analogous extensions of results known for the autonomous systems. Moreover, it is observed that not only the statements but arguments similar to the autonomous case hold good in the non-autonomous case and thus give rise to the analogous extensions. This happens due to the fact that many of the proofs for the autonomous case do not use the fact that the governing rule is constant with respect to time and hence similar techniques or methodology can be used to derive the results for the non-autonomous case. For example the proof establishing the equivalence of topological mixing for the two systems $(X,f)$ and $(\Psi,\overline{f})$ uses the fact that if every pair $(U,V)$ of non-empty open sets interacts at times $n\geq n_{U,V}$, then open sets $U_1,U_2,\ldots,U_n$ and $V_1,V_2,\ldots,V_m$ also interact for $n\geq K$ where $K=\max\{n_{U_i,V_j}: 1\leq i\leq n, 1\leq j\leq m\}$ and hence a similar technique yields equivalence of topological mixing for the two systems in the non-autonomous case. Similarly, if for each non-empty open set $U$ there exists $n_U\in\mathbb{N}$ such that diam$\omega_n(U)\geq\delta,~~\forall n\geq n_U$ then for any finite collection of non-empty open sets $U_1,U_2,\ldots,U_m$ there exists $K=\max\{n_{U_i}: i=1,2,\ldots,m\}$ such that diam$(\omega_n(U_i))\geq\delta,~~\forall n\geq K$ and $i=1,2,\ldots,m$ and hence a proof similar to the autonomous case yields equivalence of strong sensitivity for the two systems in the non-autonomous case(when the hyperspace $\mathcal{K}(X)$ is equipped with the Hausdorff metric). Such arguments does not use the fact that the governing rule is constant with time and hence also hold good for the non-autonomous systems. Similar observations can be made about the proofs involving properties like sensitivity, equicontinuity and Li-Yorke chaoticity. We now state the results for the non-autonomous systems whose autonomous version do not use constancy of $f$ and hence are trivial extensions of their autonomous versions.
\end{Remark}

\begin{p} Let $\mathcal{F} (X) \subseteq \Psi.$ If $(X,\mathbb{F})$ is topologically mixing, then so is $(\Psi, \overline{\mathbb{F}})$. The converse holds if there exists a base $\beta$ for topology on $X$ such that $U^+ \in \Delta$ for every $U \in \beta$.
\end{p}

\begin{p}
$(\mathcal{K}(X),\overline{\mathbb{F}})$ is sensitive $\Rightarrow$ $(X,\mathbb{F})$ is
sensitive.
\end{p}

\begin{p} \label{ss1}
$(\mathcal{K}(X),\overline{\mathbb{F}})$ is strongly sensitive $\Leftrightarrow$
$(X,\mathbb{F})$ is strongly sensitive.
\end{p}

\begin{Remark}
The above results analogously relate the dynamical behavior of a non-autonomous system with its induced counterpart establishing the equivalence of topological mixing and strong sensitivity for the non-autonomous case. The results also establish that if the induced system is sensitive then the original system is also sensitive. However, for the converse it is known that there exist sensitive autonomous systems $(X,f)$ such that induced system is not sensitive\cite{pa2}. We now establish existence of a sensitive non-autonomous system $(X,\mathbb{F})$ such that the induced system is not sensitive.
\end{Remark}

\begin{ex}
Let $(X,f)$ be a sensitive autonomous systems  such that induced system is not sensitive and let $I$ be the identity operator on $X$. Let $\mathbb{F}=\{f,I,f,I,\ldots,f,I,\ldots\}$. As $\omega_{2k-1}(A)=f^k(A)$ for any $A\subset X$, sensitivity of $(X,f)$ implies sensitivity of $(X,\mathbb{F})$. Further as $(\mathcal{K}(X),\overline{f})$ is not sensitive, $(\mathcal{K}(X),\overline{\mathbb{F}})$ is not sensitive and hence there exist sensitive non-autonomous dynamical systems such that their induced counterparts are not sensitive.
\end{ex}


We now give some more results which appear as trivial analogous extensions of their autonomous counterpart.

\begin{p}
$(\mathcal{K}(X),\overline{\mathbb{F}})$ is Li-Yorke sensitive $\implies$
$(X,\mathbb{F})$ has chaotic dependence on initial conditions. Further, if
$(\mathcal{F}(X),\overline{\mathbb{F}})$ is Li-Yorke sensitive then $(X,\mathbb{F})$
is Li-Yorke sensitive.
\end{p}

\begin{p}
Let $X$ be a locally connected. Then, $(X,\mathbb{F})$ is sensitive
$\Rightarrow$ $(\mathcal{F}(X), \overline{\mathbb{F}})$ is pointwise
sensitive.
\end{p}

\begin{p}
Let $\mathcal{F}(X) \subseteq \Psi \subseteq \mathcal{K} (X)$.
$(X,\mathbb{F})$ is equicontinuous $\Rightarrow$ $(\Psi, \overline{\mathbb{F}})$ is
almost equicontinuous.
\end{p}

\begin{p}
Let $\mathcal{F}_1(X) \subseteq \Psi$. Then, $(\Psi, \overline{\mathbb{F}})$
is equicontinuous $\Rightarrow$ $(X,\mathbb{F})$ is equicontinuous.
\end{p}

\begin{p}
Let $\mathcal{F}(X) \subseteq \Psi \subseteq \mathcal{K} (X)$.
$(X,\mathbb{F})$ is uniformly equicontinuous if and only if $(\Psi,
\overline{\mathbb{F}})$ is uniformly equicontinuous.
\end{p}

\begin{p} $(X,\mathbb{F})$ is uniformly equicontinuous if and only if
$(CL(X), \overline{\mathbb{F}})$ is uniformly equicontinuous.
\end{p}

\begin{p}
Let $(X,f)$ be a dynamical system and let $(\mathcal{K}(X), \overline{\mathbb{F}})$ be the induced dynamical system on the hyperspace. If $\mathcal{F}_1(X) \subseteq \Psi$, the system $(X,\mathbb{F})$ has positive topological entropy implies $(\Psi, \overline{\mathbb{F}})$ has a positive topological entropy. However, the converse is not true.
\end{p}

\begin{p}
Let $\mathcal{F}_1(X) \subseteq \Psi \subseteq  CL(X)$. Then,
$(\Psi, \overline{\mathbb{F}})$ is $\delta$-expansive implies $(X,\mathbb{F})$ is
$\delta$-expansive.
\end{p}

\begin{p} Let $\Psi$ contain the set of all singletons. If $(X,\mathbb{F})$ is Li-
Yorke chaotic, so is $( \Psi, \overline{\mathbb{F}})$. However, the converse is not true.
\end{p}

\section{Conclusion}
In this work, we have established relation between the dynamical behavior of a non-autonomous system and its induced counterpart. We have established that while properties like topologically mixing, strong sensitivity and uniform equicontinuity are equivalent for the two systems $(X,\mathbb{F})$ and $(\Psi,\overline{\mathbb{F}})$, weak mixing is equivalent for the two systems when the family $\mathbb{F}$ is commutative. We establish that if induced system is sensitive or transitive then the original system exhibits similar dynamical behavior. It is observed that the results obtained for the non-autonomous case are analogous extensions of their autonomous versions. It is further observed that many of the proofs for the autonomous version do not use the autonomous nature of the system and hence many of the results for the autonomous case extend naturally to the non-autonomous case. Further, as commutativity of the family $\mathbb{F}$ establishes equivalence of weak mixing for the two systems, it is expected that if the non-autonomous system is "nice enough", many of the results from the autonomous case can be extended to the non-autonomous version. As a result one may consider properties like distality, minimality, syndetic sensitivity and other dynamical notions for further investigation. It may also be possible to improve(under additional natural conditions) results for properties like Li-Yorke sensitivity, Li-Yorke chaoticity and topological entropy. However we do not consider these questions in this article and leave them open for further investigation. It is emphasised that if the system is "nice enough", the dynamical behavior between the two systems does not change drastically when compared to the autonomous version. The result is an indicator of the qualitative stability of the long term behavior as any autonomous(non-autonomous) system can be approximated using non-autonomous(autonomous) system.

\bibliography{xbib}

\begin{thebibliography}{99}

\bibitem{ba}
{\bf  Banks John,} Chaos for induced hyperspace maps, {\it Chaos,
Solitons and Fractals,} 25(2005) 681-685.


\bibitem{be}
{\bf Beer G,} Topologies on Closed and Closed Convex Sets, {\it
Kluwer Academic Publishers, Dordrecht/Boston/London} (1993).

\bibitem{bc}
{\bf Block L, Coppel W,} Dynamics in one dimension, {\it
Springer-Verlag, Berlin Hiedelberg}  (1992).

\bibitem{bs}
{\bf Brin Michael, Stuck Garrett,} Introduction to dynamical
systems, {\it Cambridge University Press} (2002).

\bibitem{de}
{\bf Devaney Robert L,} Introduction to chaotic dynamical systems,
{\it Addison Wesley} (1986).

\bibitem{klaus}
{\bf Klaus R, Rohde P,} Fuzzy chaos: Reduced chaos in the combined dynamics of
several independently chaotic populations, {\it The American Naturalist} 158, no. 5 (2001),
553–556.
%

%
%

\bibitem{mi}
{\bf Michael E,} Topologies on spaces of subsets, {\it Trans. Amer.
Math. Soc.,} 71(1951), 152-82.

\bibitem{nai}
{\bf Naimpally S,} All Hypertopologies are hit-and-miss, {\it Appl. Gen. Topol.} 3, no. 1 (2002),
45–53.

\bibitem{seb}
{\bf Sebastien D, Huw D,} Combined dynamics of boundary and interior perturbations
in the Eady setting, {\it Journal of the Atmospheric Sciences} 61, no. 13 (2004), 1549–1565.


\bibitem {pa1}
{\bf Sharma Puneet, Nagar Anima,}  Topological Dynamics on
Hyperspaces, {\it Appl. Gen. Topol.} 11, (2010), no. 1, 1-19.

\bibitem {pa2}
{\bf Sharma Puneet, Nagar Anima,}  Inducing Sensitivity on
Hyperspaces, {\it Topology and its Applications} 157, (2010),
2052-2058.

\bibitem {pm}
{\bf Sharma Puneet, Raghav Manish,}  Dynamics of Non-Autonomous Discrete Dynamical Systems,
 {\it Topology Proceedings}(Accepted).

\bibitem{str}
{\bf Strogatz H,} Nonlinear Dynamics And Chaos: With Applications To Physics, Biology, Chemistry, And Engineering, {\it Westview Press} (2001).  

\bibitem{yang}
{\bf Yang Z, Satoshi Y, Guanhua C,} Reduced density matrix and combined dynamics
of electron and nuclei, {\it Journal of Chemical Physics} 13, no. 10 (2000), 4016–4027.
\end{thebibliography}

\end{document}